\documentclass[12pt, fleqn]{amsart}

\usepackage{amssymb}
\usepackage{amsthm}
\usepackage{amsmath}
\usepackage{amsaddr}
\usepackage{fontenc}
\usepackage{inputenc}
\usepackage{enumitem}
\usepackage{hyperref}
\usepackage{xcolor}

\usepackage{amsrefs}
\usepackage{mathrsfs}
\usepackage{verbatim}

\usepackage[top = 1in, left = 1.25in, right = 1in, bottom = 1in]{geometry}

\usepackage{tikz}
\usetikzlibrary{scopes,arrows,decorations.pathmorphing,decorations.markings,backgrounds,positioning,fit,petri}
\tikzstyle{line} = [draw, thick, -latex']
\tikzstyle{bigArrow} = [thick, decoration={markings,mark=at position
   1 with {\arrow[semithick]{open triangle 60}}},
   double distance=1.4pt, shorten >= 5.5pt,
   preaction = {decorate},
   postaction = {draw,line width=1.4pt, white,shorten >= 4.5pt}]
\tikzstyle{innerWhite} = [semithick, white,line width=1.4pt, shorten >= 4.5pt]
\definecolor{mycolor}{RGB}{0,160,0}

\def\real{\hbox{\rm\setbox1=\hbox{I}\copy1\kern-.45\wd1 R}}
\def\natural{\hbox{\rm\setbox1=\hbox{I}\copy1\kern-.45\wd1 N}}

\newcommand{\be}{\begin{equation}}
\newcommand{\en}{\end{equation}}

\theoremstyle{definition}

\newtheorem*{ack*}{Acknowledgment}

\newtheorem{thm}{Theorem}
\newtheorem{prop}[subsection]{Proposition}
\newtheorem{lem}[subsection]{Lemma}

\newtheorem{cor}[subsection]{Corollary}

\theoremstyle{remark}

\numberwithin{equation}{section}

\newcommand{\B}{\mathcal{B}}

\title[Slow Entropies for Generic Transformations]
{Genericity and Rigidity for Slow Entropy Transformations}

\author{Terry Adams}
\address{In memory of mathematician and artist Nat Friedman.}
\email{terry@ieee.org}

\date{}

\begin{document}

\begin{abstract}
The notion of slow entropy, both upper and lower slow entropy, was defined 
by Katok and Thouvenot as a more refined measure of complexity 
for dynamical systems, than the classical Kolmogorov-Sinai entropy.  
For any subexponential rate function $a_n(t)$, we prove there exists a generic class 
of invertible measure preserving systems such that the lower slow entropy is zero and 
the upper slow entropy is infinite. 
Also, given any subexponential rate $a_n(t)$, 
we show there exists a rigid, weak mixing, invertible system such that 
the lower slow entropy is infinite with respect to $a_n(t)$. 
This gives a general solution to a question on the existence 
of rigid transformations with positive polynomial upper slow entropy,  
Finally, we connect slow entropy with the notion of entropy covergence rate 
presented by Blume.  In particular, we show slow entropy is a strictly 
stronger notion of complexity and give examples which have zero upper slow entropy, 
but also have an arbitrary sublinear positive entropy convergence rate.
\end{abstract}


\maketitle


\section{Introduction}
The notion of slow entropy was introduced by Katok and Thouvenot in \cite{katok1997slow} 
for amenable discrete group actions.  It generalizes the classical notion 
of Kolmogorov-Sinai entropy \cite{Kolmogorov, Sinai2} 
for Z-actions and gives a method for distinguishing the 
complexity of transformations with zero Kolmogorov-Sinai entropy\footnote{The Kolmogorov-Sinai entropy 
of a transformation $T$ is referred to as the entropy of $T$.}.  
The recent survey \cite{KKW2020} gives a general account of several extensions of entropy, 
including a comprehensive background on slow entropy. 
Slow entropy has been computed for several examples including compact group rotations, 
Chacon-3 \cite{Fr70}, the Thue-Morse system and the Rudin-Shapiro system.  
In \cite{ferenczi1997measure}, it is shown that the lower slow entropy of any rank-one transformation 
is less than or equal to 2.  Also, in \cite{ferenczi1996rank},  
it is shown there exist rank-one transformations 
with infinite upper slow entropy with respect to any polynomial. 
In \cite{kanigowski2018slow}, Kanigowski is able to get more precise upper bounds on slow entropy 
of local rank-one flows.  Also, in \cite{kanigowski2019slow}, the authors obtain polynomial 
slow entropies for unipotent flows. 

In \cite{KKW2020}, the following question is given:

\noindent
Question 6.1.2. Is it possible to have the upper slow entropy for a rigid transformation 
positive with respect to $a_n(t) =n^t$?

\noindent 
We give a positive answer to this question.  Given any subexponential rate, 
we show that a generic transformation has infinite upper slow entropy with respect to that rate.  
We say $a_n(t) > 0$, for $n \in \natural$ and $t > 0$, is subexponential, if given $\beta > 1$ 
and $t > 0$, $\lim_{n\to \infty} \frac{a_n(t)}{\beta^n} = 0$.  
We will only consider monotone $a_n(t)$ such that $a_n(t) \geq a_n(s)$ for $t > s$. 
Let $(X, \B, \mu )$ be a standard probability space (i.e., isomorphic 
to $[0,1]$ with Lebesgue measure).  Also, let 
\[
\mathcal{M} = \{ T:X\to X\ |\ T\ \mbox{is invertible and preserves $\mu$} \} . 
\]
We consider the weak topology on the space $\mathcal{M}$ which is induced 
by the strong operator topology on the space of Koopman operators, 
$\{ U_T : T \in \mathcal{M} \}$.  
One of our three main results is the following. 
\begin{thm}
Let $a_n(t)$ be any subexponential rate function. 
There exists a dense $G_{\delta}$ subset $G \subset \mathcal{M}$ such that for each 
$T \in G$, the upper slow entropy of $T$ is infinite with respect to $a_n(t)$. 
\end{thm}
Thus, the generic transformation answers question 6.1.2 in the affirmative, 
since the generic transformation is known to be weak mixing and rigid.  
Our proof is constructive and provides a recipe for constructing 
rigid rank-ones with infinite upper slow entropy. 

We show that there is a generic class of transformations such that 
the lower slow entropy is zero with respect to a given divergent rate. 
\begin{thm}
Suppose $a_n(t) \in \real$ is a rate such that for $t >0$, 
$\lim_{n\to \infty} a_n(t) = \infty$.  
There exists a dense $G_{\delta}$ subset $G \subset \mathcal{M}$ such that for each 
$T \in G$, the lower slow entropy of $T$ is zero with respect to $a_n(t)$. 
\end{thm}
This shows for any slow rate $a_n(t)$, the generic transformation 
has infinitely occurring time spans where the complexity is sublinear. 
This is due to ``super'' rigidity times for a typical tranformation. 
This raises the question of whether there exists an invertible 
rigid measure preserving transformation with infinite polynomial 
lower slow entropy.  We answer this question by constructing examples 
with infinite subexponential lower slow entropy in section \ref{rigid-inf-slow-ent-0}. 
This also answers question 6.1.2. 
\begin{thm}
There exists a family $\mathcal{F} \subset \mathcal{M}$ of rigid, weak mixing transformations 
such that given any subexponential rate $a_n(t)$, there exists a transformation 
in $\mathcal{F}$ which has infinite lower slow entropy with respect to $a_n(t)$. 
\end{thm}
In the final section, we give the connections with entropy convergence rate 
as defined by Frank Blume in \cite{blume2012relation}. 

\section{Preliminaries}
We describe the setup and then give a few lemmas used in the proofs 
of our main results. 

\subsection{Definitions}
Given an alphabet $\alpha_1, \alpha_2, \ldots , \alpha_r$, a codeword of length $n$ 
is a vector $w = \langle w_1, w_2, \ldots , w_{n} \rangle$ 
$= \langle w_i \rangle_{i=1}^{n}$ 
such that $w_i \in \{ \alpha_1, \ldots , \alpha_r \}$ for $1 \leq i \leq n$.  
Our codewords will be obtained from a measure preserving system 
$(X, \B, \mu, T)$ and finite partition $P = \{ p_1, p_2, \ldots , p_r \}$.  
In this case, we will consider the alphabet to be $\{ 1, \ldots , r \}$.  
Given $x \in X$ and $n \in \natural$, define the codeword 
$\vec{P}_n (x) = \langle w_i \rangle_{i=1}^{n}$ such that 
$T^{i-1} x \in p_{w_i}$.  When using this notation, the transformation will be fixed. 

Let $w, w^{\prime}$ be codewords of length $n$. 
The (normalized) Hamming distance is defined as:
\[
d(w, w^{\prime}) = \frac{1}{n} \sum_{i=1}^{n} \big( 1 - \delta_{w_i w_i^{\prime}} \big) . 
\]
Given a codeword $w$ of length $n$ and $\varepsilon > 0$, an $\varepsilon$-ball is the subset 
$V \subseteq \{ 1,\ldots ,r \}^{n}$ such that 
$d(w,v) < \varepsilon$ for $v \in V$. 
We will denote the $\varepsilon$-ball as $B_{\varepsilon}(w)$.  If given a transformation $T$ 
and partition $P$, define 
\[
B_{\varepsilon}^{T,P}(w) = \{ x\in X: d( \vec{P}_{n} (x) , w ) < \varepsilon \} . 
\]
Given $\varepsilon > 0$, $\delta > 0$, $n\in \natural$, 
finite partition $P=\{ p_1, p_2, \ldots , p_r \}$ and 
dynamical system $(X, \B, \mu, T)$, define $S_P(T,n,\varepsilon, \delta) = S$ as: 
\[
S = \min{\{ k : \exists v_1, \ldots , v_k \in \{1,\ldots,r\}^{n}\ 
\mbox{such that}\ \mu \big( \bigcup_{i=1}^{k} B_{\varepsilon}^{T,P}(v_i) \big) \geq 1 - \delta \}} . 
\]

Now we give the definition of upper and lower slow entropy 
for $\mathbb{Z}$-actions.  For more general discrete amenable group 
actions, the interested reader may see the survey \cite{KKW2020}.  
Also, in \cite{hochman2012slow}, slow entropy is used to construct 
infinite-measure preserving $\mathbb{Z}^2$-actions 
which cannot be realized as a group of diffeomorphisms 
of a compact manifold preserving a Borel measure.  
Let $T$ be an invertible measure preserving transformation 
defined on a standard probability space $(X, \B, \mu)$. 
Let $a=\{ a_n(t) : n \in \natural, t>0 \}$ 
be a family of positive sequences monotone in $t$ and such that 
$\lim_{n\to \infty} a_n(t) = \infty$ for $t > 0$.  
Define the upper (measure-theoretic) slow entropy of $T$ with respect 
to a finite partition $P$ as 
\begin{align*}
\mbox{s-} \overline{\mathcal{H}}_{a}^{\mu} (T,P) = &
\lim_{\delta\to 0}\lim_{\varepsilon\to 0} \ \mbox{s-}\overline{\mathcal{H}} 
\big( \varepsilon, \delta, P \big) , \\ 
\mbox{where }\ \mbox{s-}\overline{\mathcal{H}} 
\big( \varepsilon, \delta, P \big) = & 
\left\{\begin{array}{ll}
\sup{ \overline{\mathcal{G}}(\varepsilon, \delta, P) }, & \mbox{if }\ 
 \overline{\mathcal{G}} 
\big( \varepsilon, \delta, P \big)\neq \emptyset , \\ 
0, & \mbox{if }\ \overline{\mathcal{G}} 
\big( \varepsilon, \delta, P \big) = \emptyset , 
\end{array}
\right. \\ 
\mbox{and } \ \overline{\mathcal{G}} \big( \varepsilon, \delta, P \big) = & 
\ \{ t > 0 : \limsup_{n\to\infty} \frac{S_P(T,n,\varepsilon,\delta)}{a_n(t)} > 0 \} . 
\end{align*}
The upper slow entropy of $T$ with respect to $a_n(t)$ is defined as 
\[
\mbox{s-} \overline{\mathcal{H}}_{a}^{\mu} (T) = 
\sup_{P} \ \mbox{s-} \overline{\mathcal{H}}_{a}^{\mu} (T,P) . 
\]
To define the lower slow entropy of $T$, replace $\limsup$ 
in the definition above with $\liminf$.  

\subsection{Supporting lemmas}
Define the binary entropy function, 
\[
\mathcal{H} ( x ) = - x \log_2{(x)} - \big( 1 - x \big) \log_2{(1-x)} . 
\]
We give some preliminary lemmas involving binary codewords and measurable partitions 
that are used in the main results. 
\begin{lem}
\label{ham1}
Suppose $w_1, w_2$ are binary words of length $\ell$ with Hamming distance $d = d(w_1, w_2) > 0$. 
Let $m \in \natural$ and $C$ be the set of all $2^m$ codewords consisting 
of all possible sequences of words $w$ from $\{ w_1, w_2 \}$ of length $n=m\ell$. 
Given $\varepsilon, \theta > 0$ with $\frac{\varepsilon}{d} + \frac{1}{m} < \frac{1}{2}$, 
the minimum number $S$ of $\varepsilon$-balls required to cover $1-\theta$ 
of the words in $C$ satisfies: 
\[
S \geq \big( 1 - \theta \big) 
2^{m \big( 1-\mathcal{H}( \frac{2\varepsilon}{d} + \frac{1}{m} ) \big) } . 
\]
\end{lem}
\begin{proof}
The proof follows from a standard bound on the size of Hamming balls 
\cite{macwilliams1977theory} (p.310).  
Suppose $v_1, \ldots , v_j$ are a minimum number of centers such that 
$\varepsilon$-balls $B_{\varepsilon}(v_i)$ cover at least $1 - \theta$ 
of codewords in $C$.  For each $i$, choose $u_i \in B_{\varepsilon}(v_i)$.  
Thus, $B_{2\varepsilon}(u_i) \supseteq B_{\varepsilon}(v_i)$ and 
the $2\varepsilon$-balls $B_{2\varepsilon}(u_i)$ cover at least 
$1 - \theta$ of the codewords in $C$.  

This reduces the problem to a basic Hamming ball size question.  
Since all words are generated by $w_1, w_2$, we can map $w_1$ to 0 and $w_2$ to 1, 
and consider the number of Hamming balls needed to cover $1-\theta$ 
of all binary words of length $m$. 
Thus, if at least 
$\lceil \frac{2\varepsilon}{d} m \rceil$ 
words differ, then the distance is greater than or equal to $2\varepsilon$.  
Also, $\lceil \frac{2\varepsilon}{d} m \rceil \leq m 
\big( \frac{2\varepsilon}{d} + \frac{1}{m} \big)$.  
By \cite{macwilliams1977theory}(p.310), a Hamming ball of radius 
$\frac{2\varepsilon}{d}$ has a volume less than or equal to:
\[
2^{ m \mathcal{H}( \frac{2\varepsilon}{d} + \frac{1}{m} ) } . 
\]
Therefore, the minimum number of balls required to cover at least $(1-\theta)$ of the space is: 
\[
\big( 1 - \theta \big) 2^{ m \big( 1 - \mathcal{H}( \frac{2\varepsilon}{d} + \frac{1}{m} ) \big) } . 
\]
\end{proof}

\begin{lem}
\label{ham2}
Suppose the setup is similar to Lemma \ref{ham1} and there are two generating words 
$w_1, w_2$ of length $\ell$ with distance $d = d(w_1,w_2)$.  
Suppose $C$ is the set of $2^m$ codewords consisting of all possible 
sequences of blocks of either $w_1$ or $w_2$.  
Let $b\in \real$ such that $0 < b \leq 1$. 
Define $\phi : A \to C$ and measure $\mu$ 
such that $\mu ( \{ x\in A : \phi(x)=v \} ) = \frac{b}{2^m}$ 
for $v \in C$.  
Suppose $\psi: A \to C$ is a map 
satisfying:  
\[
\mu \Big( \{ x \in A: d ( \psi(x), \phi(x) ) < \eta \} \Big) > \big( 1 - \eta \big) b . 
\]
The minimum number $S$ of $\varepsilon$-Hamming balls $B$ such that 
\[
\mu \big( \{ x\in A : \psi(x) \in B \} \big) \geq 1 - \theta 
\]
satisfies 
\[
S  \geq \big( 1 - \theta - \eta \big) 2^{m(1 - \mathcal{H}(\frac{2(\varepsilon + \eta)}{d} + \frac{1}{m}))} . 
\]
\end{lem}
\begin{proof}
Let $E = \{ x \in A: d ( u(x), v(x) ) \geq \eta \}$.  
For $x \in A\setminus E$, $B_{\varepsilon}(\psi(x)) \subseteq B_{\varepsilon+\eta}(\phi(x))$.  
By Lemma \ref{ham1}, 
\[
\Big( 1 - \theta - \eta \Big) 2^{m( 1 - \mathcal{H}(\frac{2(\varepsilon+\eta)}{d}+\frac{1}{m}))} 
\]
$(\varepsilon + \eta)$-balls are needed to cover $1 - \theta - \eta$ of $\phi(x)$ words.  
Thus, the total number of $\varepsilon$-balls 
needed to cover $(1 - \theta)$ mass of $\psi(x)$ words is at least:
\[
\Big( 1 - \theta - \eta \Big) 2^{m( 1 - \mathcal{H}(\frac{2(\varepsilon+\eta)}{d}+\frac{1}{m}))} . 
\]
\end{proof}
The following lemma is used in the proof of Proposition \ref{dense-infinity-prop}. 
\begin{lem}
\label{ham4}
Let $\eta > 0$ and $r,n \in \natural$.  
Let $(X, \B, \mu, T)$ be an invertible measure preserving system 
and $b$ a set of positive measure such that $\hat{b} = \bigcup_{i=0}^{n-1} T^i b$ 
is a disjoint union (except for a set of measure zero).  
Suppose $P = \{ p_1, p_2, \ldots , p_r \}$ 
and $Q = \{ q_1, q_2, \ldots , q_r \}$ are partitions such that 
\begin{align}
\sum_{i=1}^{r} \mu \big( (p_i \cap \hat{b}) \triangle (q_i \cap \hat{b}) \big) 
< \eta^2 \mu(\hat{b}) . \label{part-comp-eqn}
\end{align}
Then for $n \in \natural$, 
\[
\mu \big( \{ x \in b : d ( \vec{P}_n ( x ) , \vec{Q}_n ( x ) ) < \eta \} \big) 
> \big( 1 - \eta \big) \mu(b) . 
\]
\end{lem}
\begin{proof}
Define 
\[
R_n = \bigvee_{i=0}^{n-1} T^{-i} \big( P \vee Q \big) . 
\]
Define 
\[
A = \{ p\in R_n\cap b : \# \{ i : 0\leq i < n, 
T^i p \subset \bigcup_{j=1}^{r} \big( p_j \cap q_j \big) \} \leq (1-\eta)n \} . 
\]
We show $\mu (A) < \eta \mu (b)$.  Otherwise, for $p\in A$ and $i$ such that 
$T^i p \subseteq p_j \cap q_k$ for $j \neq k$, this contributes $2\mu(p)$ 
to the sum (\ref{part-comp-eqn}). 
Thus, for $p \in A$, the number of such $i$ gives measure greater than or equal to 
$2\mu(p) \eta n$.  Adding up over all $p\in A$ gives measure greater than or equal to 
$2\eta \mu(b) \eta n > \eta^2 \mu (\hat{b})$.  
For a.e. $x,y \in b\cap A^c$, $d ( \vec{P}_n ( x ) , \vec{Q}_n ( x ) ) < \eta$ 
and this holds for $\mu (b \cap A^c) > (1-\eta)\mu(b)$.  
\end{proof}
The following lemma is a more general version of Lemma \ref{ham4} and 
used in multiple places throughout this paper. 
Given two ordered partitions $P = \langle p_1, p_2, \ldots , p_r \rangle$ and 
$Q = \langle q_1, q_2, \ldots , q_r \rangle$, let 
\[
D(P,Q) = \sum_{i=1}^{r} \mu ( p_i \triangle q_{i} ) . 
\]
\begin{lem}
\label{ham5}
Let $(X, \B, \mu, T)$ be ergodic. 
Let $\eta > 0$ and $r,n \in \natural$.  Suppose $P = \langle p_1, p_2, \ldots , p_r \rangle$ 
and $Q = \langle q_1, q_2, \ldots , q_r \rangle$ are ordered partitions such that 
\begin{align}
D(P,Q) < \eta^2 . \label{part-comp-eqn2}
\end{align}
Then for $n \in \natural$, 
\[
\mu \big( \{ x: d ( \vec{P}_n ( x ) , \vec{Q}_n ( x ) ) < \eta \} \big) > 1 - \eta . 
\]
\end{lem}
\begin{proof}
Let $\eta_1 \in \real$ such that 
\begin{align}
\sum_{i=1}^{r} \mu ( p_i \triangle q_i ) < \eta_1^2 < \eta^2 . \label{part-comp-eqn2}
\end{align}
Define 
\[
R_n = \bigvee_{i=0}^{n-1} T^{-i} \big( P \vee Q \big) . 
\]
Let $\eta_0 < \frac{1}{2}$ and $C = \{ I_0, \ldots , I_{n-1} \}$ be a Rohklin tower 
such that $\mu ( \bigcup_{i=0}^{n-1} I_i ) > 1 - \frac{\eta_0}{n}$.  
Define 
\[
A_0 = \{ p\in R_n\cap I_0 : \# \{ i : 0\leq i < n, 
T^i p \subset \cup_{j=1}^{r} \big( p_j \cap q_j \big) \} \leq (1-\eta_1)n \} . 
\]
We show $\mu (A_0) < \eta_1 \mu (I_0)$.  Otherwise, for each $i$ such that 
$T^i p \subseteq p_j \cap q_k$ for $j \neq k$, this contributes $2\mu(p)$ 
to the sum (\ref{part-comp-eqn2}). 
Thus, for $p \in A_0$, the number of such $i$ gives measure greater than 
$2\mu(p) \eta_1 n$.  Adding up over all $p\in A_0$ gives measure greater than 
$2\eta_1 \mu(I_0) \eta_1 n > 2 \eta_1^2 (1 - \eta_0)$.  
For $x,y \in I_0\cap A_0^c$, $d ( \vec{P}_n ( x ) , \vec{Q}_n ( x ) ) < \eta_1$ 
and this holds for $\mu (I_0 \cap A_0^c) > (1-\eta_1)\mu(I_0)$.  
By showing the analogous result for $A_k$ defined as: 
\[
A_k = \{ p\in R_n\cap I_k : \# \{ i : 0\leq i < n, 
T^i p \subset \cup_{j=1}^{r} \big( p_j \cap q_j \big) \} \leq (1-\eta_1)n \} , 
\]
then $\mu (I_k \cap A_k^c) > (1-\eta_1)\mu(I_k)$. 
Hence, 
\[
\sum_{k=0}^{n-1} \mu (I_k \cap A_k^c) > (1 - \eta_1)(1 - \eta_0) . 
\]
Therefore, since $\eta_0$ may be chosen arbitrarily small, our claim holds. 
\end{proof}


\subsection{Infinite rank}
A result of Ferenczi \cite{ferenczi1997measure} shows that the lower slow entropy of a rank-one 
transformation is less than or equal to 2 with respect to $a_n(t)=n^t$.  
Thus, our examples in section \ref{rigid-inf-slow-ent} are not rank-one 
and instead, have infinite rank. We will adapt the technique 
of independent cutting and stacking to construct 
rigid transformations with infinite lower slow entropy. 
Independent cutting and stacking was originally defined in 
\cite{friedman1972mixing, shields1973cutting}.  A variation of this technique is used 
in \cite{katok1997slow} to obtain different types of important counterexamples.  
For a general guide on the cutting and stacking technique, see \cite{Fri92}.

\section{Generic class with zero lower slow entropy}
Let $a_n(t)$ be a sequence of real numbers such that $a_n(t) \geq a_n(s)$ for $t>s$ 
and $\lim_{n\to \infty} a_n(t) = \infty$ for $t > 0$. 
For $N,t, M \in \natural$ and any finite partition $P$, define 
\begin{align}
\underline{G}(N,t, M, P) = \{ T \in \mathcal{M} : \exists n > N, 0 < \delta < \frac{1}{M} 
\ \mbox{such that}\ S_{P}(T, n, \delta, \delta) < \frac{a_n(\frac{1}{t})}{N} \} . 
\end{align}

\begin{prop}
\label{open-zero-prop}
For $N,t,M \in \natural$ and finite partition $P$, 
the set $\underline{G}(N,t,M,P)$ is open in the weak topology on $\mathcal{M}$. 
\end{prop}
\begin{proof}
Let $T_0 \in \underline{G}(N,t,M,P)$ and $n>N$, $0 < \delta_0 < \frac{1}{M}$ be such that 
$S_{P}(T, n, \delta_0, \delta_0) < \frac{a_n(\frac{1}{t})}{N}$. 
Let $P_n = \bigvee_{i=0}^{n-1} T_0^{-i} P$.  
Choose $\delta_1 \in \real$ such that $\delta_0 < \delta_1 < \frac{1}{M}$ .  
Let $\alpha = {\big( \delta_1 - \delta_0 \big)}/{2}$.  
In the weak topology, 
choose an open set $U$ containing $T_0$ such that for $T_1 \in U$, 
$0 \leq i < n$, and $p \in P_n$, 
\begin{align}
\mu ( T_1 T_0^i p \triangle T_0^{i+1}p ) &\leq \big( \frac{\alpha}{n^2} \big) \mu ( p ) \label{eqn1} .
\end{align}
We will prove inductively in $j$ for $p \in P_n$ that 
\begin{align}
\mu ( T_0^{j} p \triangle T_1^{j} p ) &\leq \big( \frac{j\alpha}{n^2} \big) \mu ( p ) . \label{eqn2}
\end{align}
The case $j=1$ follows directly from (\ref{eqn1}): 
\begin{align*}
\mu ( T_0 p \triangle T_1 p ) &\leq \big( \frac{\alpha}{n^2} \big) \mu (p) . 
\end{align*}
Also, the case $j=0$ is trivial. 
Suppose equation (\ref{eqn2}) holds for $j = i$.  Below shows it holds for $j = i+1$: 
\begin{align*}
\mu ( T_0^{i+1} p \triangle T_1^{i+1} p ) &\leq \mu ( T_0^{i+1} p \triangle T_1 T_0^{i} p ) + 
\mu ( T_1 T_0^{i} p \triangle T_1^{i+1} p ) \\ 
&= \mu ( T_0^{i+1} p \triangle T_1 T_0^{i} p ) + \mu ( T_0^{i} p \triangle T_1^{i} p ) \\   
&\leq  \big( \frac{\alpha}{n^2} \big) \mu ( p ) + \big( \frac{i \alpha}{n^2} \big) \mu ( p ) \\ 
&= \big( \frac{(i+1)\alpha}{n^2} \big) \mu ( p ) . 
\end{align*}
For $p \in P_n$, let 
\[
E_p = \bigcap_{i=0}^{n-1} T_1^{-i} T_0^{i} p . 
\]
Thus, 
\begin{align}
\mu \big( E_p \big) &\geq \mu (p) - \sum_{i=1}^{n-1} \mu (p \triangle T_1^{-i}T_0^{i} p) \\ 
&= \mu (p) - \sum_{i=1}^{n-1} \mu (T_1^{i} p \triangle T_0^{i} p) \\
&> (1 - \alpha) \mu (p) . 
\end{align}
Hence, if $E = \bigcup_{p \in P_n} E_p$, $\mu(E) > 1 - \alpha$.  
Each $x \in E$ has the same $P$-name under $T_1$ and $T_0$.  
Suppose $V\subseteq 2^{\{0,1\}^n}$ is such that $A_0 = \bigcup_{v\in V} B^{T_0}_{\delta_0}(v)$ 
satisfies $\mu ( A_0 ) \geq 1 - \delta_0$ and 
\[
\mbox{card}(V) < \frac{a_n(\frac{1}{t})}{N} . 
\]
Let $A_1 = \bigcup_{v\in V} B^{T_1}_{\delta_1}(v)$ and 
$A_1^{\prime} = \bigcup_{v\in V} B^{T_1}_{\delta_0}(v)$.  
Since $\mu ( A_0 \triangle A_1^{\prime} ) \leq \mu(E^c) < \alpha$, then 
\[
\mu (A_1) \geq \mu (A_1^{\prime}) > 1 - \delta_0 - \alpha > 1 - \delta_1 . 
\]
Therefore, since $\mbox{card}(V) < \frac{a_n({1}/{t})}{N}$, 
$S_{P}(T_1, n, \delta_1, \delta_1) < \frac{a_n(\frac{1}{t})}{N}$ and we are done. 
\end{proof}

Now we prove the density of the class $\underline{G}(N,t,M,P)$. 
\begin{prop}
\label{dense-zero-prop}
For $N,t,M \in \natural$ and finite partition $P$, 
the set $\underline{G}(N,t,M,P)$ is dense in the weak topology on $\mathcal{M}$. 
\end{prop}
\begin{proof}
Let $P = \{ p_1, p_2, \ldots , p_r \}$ be the partition into $r$ elements 
for $r \in \natural$.  We can discard elements with zero measure. 
Since rank-ones are dense in $\mathcal{M}$, let $T_0 \in \mathcal{M}$ be a rank-one transformation 
and let $\epsilon > 0$.  
Let $\delta < \frac{1}{M}$ and define $\eta = \min{\{ \delta^2 , \epsilon \}}$.  
Choose a rank-one column $\mathcal{C} = \{ I_0, I_1, \ldots , I_{h-1} \}$ for $T_0$ such that 
\begin{enumerate}
\item $\mu ( \bigcup_{i=0}^{h-1} I_i ) > 1 - \frac{\eta}{2}$, 
\item $h > \frac{2}{\eta}$, 
\item there exist disjoint collections $J_i$ such that 
$\mu ( p_i \triangle \bigcup_{j\in J_i} I_j ) < \frac{\eta}{4r} \mu (p_i)$.  
\end{enumerate}
Let $q_i = \bigcup_{j\in J_i} I_j$ and $Q = \{ q_1, q_2, \ldots , q_r \}$.  
Now we show how to construct a transformation $T_1 \in \underline{G}(N,t,M,P)$.  
Since $T_1$ will differ by $T_0$ inside the top level or outside the column, then 
$T_1$ will be within $\epsilon$ of $T_0$.  
Choose $k_1 \in \natural$ such that $k_1 h > N$ and for $n = k_1 h$, 
\[
a_n(\frac{1}{t}) > N h . 
\]
Choose $k_2 \in \natural$ such that $k_2 > \frac{2}{\eta}$.  
Cut column $\mathcal{C}$ into $k_1k_2$ columns of equal width and stack from left to right.  
Call this column $\mathcal{C}^{\prime}$ which has height $k_1 k_2 h$.  
Let $A_1 = \{ x: d ( \vec{P}_n ( x ) , \vec{Q}_n ( x ) ) < \frac{\delta}{2} \}$.  
By Lemma \ref{ham5}, since 
\[
\sum_{i=1}^{r} \mu ( p_i \triangle q_i ) < \frac{\eta}{4} \leq \frac{\delta^2}{4}, 
\]
then 
\[
\mu \big( A_1 \big) > 1 - \frac{\delta}{2} . 
\]
Let $A_2$ be the union of levels in $\mathcal{C}^{\prime}$ except for the top $n$ levels. 
For $x \in A_2$, $\vec{Q}_n(x)$ gives at most $h$ distinct vectors.  
Also, $\delta$-balls centered at these words will cover 
$A_1 \cap A_2$.  
Precisely, 
\[
\bigcup_{x \in A_2} B_{\delta}^{T_1,Q} \big( \vec{Q}_n(x) \big) \supseteq A_1 \cap A_2 . 
\]
Since $\mu(A_1 \cap A_2) > 1 - \delta$, then 
$S_{P}(T_1,n,\delta,\delta) \leq h < \frac{a_n({1}/{t})}{N}$.  
Therefore, we are done.
\end{proof}

\begin{thm}
Suppose $a_n(t) \in \real$ is such that for $t > s$, $a_n(t) \geq a_n(s)$ and for $t >0$, 
$\lim_{n\to \infty} a_n(t) = \infty$.  
There exists a dense $G_{\delta}$ subset $G \subset \mathcal{M}$ such that for each 
$T \in G$, the lower slow entropy of $T$ is zero with respect to $a_n(t)$. 
\end{thm}
\begin{proof}
Let $P_L$ be a sequence of nontrivial measurable partitions such that for each $k \in \natural$, 
the collection $\{ P_L : L \in \natural \}$ is dense in the class of all 
measurable partitions with $k$ nontrivial elements.  
By Proposition \ref{dense-zero-prop}, for $N,t,M,L \in \natural$, 
the set $\underline{G}(N,t,M,P_L)$ is dense, and also open by Proposition \ref{open-zero-prop}.  
Thus, 
\[
G = \bigcap_{L=1}^{\infty} \bigcap_{t=1}^{\infty} \bigcap_{M=1}^{\infty} \bigcap_{N=1}^{\infty} 
\underline{G}(N,t,M,P_L) 
\]
is a dense $G_{\delta}$.  
Given a nontrivial measurable partition $P$ and $t, M \in \natural$, 
choose $L \in \natural$ such that 
\[
D(P,P_L) < \frac{1}{9M^2} . 
\]
For $T \in \bigcap_{N=1}^{\infty} \underline{G}(N,t,3M,P_L)$, 
\[
\liminf_{n\to \infty} \frac{ S_{P_L}(T,n,\frac{1}{3M},\frac{1}{3M}) }{a_n(\frac{1}{t})} = 0 . 
\]
By Lemma \ref{ham5}, 
$S_{P}(T,n,\frac{1}{M},\frac{1}{M}) \leq S_{P_L}(T,n,\frac{1}{3M},\frac{1}{3M})$.  
Therefore, for $T\in G$, the lower slow entropy is zero with respect to $a_n(t)$. 
\end{proof}

\begin{cor}
In the weak topology, the generic transformation in $\mathcal{M}$ is rigid, weak mixing, 
rank-one and has zero polynomial lower slow entropy. 
\end{cor}

\section{Generic class with infinite upper slow entropy}
The transformations in this section are constructed 
by including alternating stages of cutting and stacking.  
Suppose $T$ is representated by a single Rokhlin column 
$\mathcal{C}$ of height $h$. 

\subsection{Two approximately independent words}
\label{two-ind-section}
Cut column $\mathcal{C}$ into two subcolumns $\mathcal{C}_1$ and $\mathcal{C}_2$ of equal width.  
Given $k \in \natural$, cut $\mathcal{C}_1$ into $k$ subcolumns of equal width, 
stack from left to right, and place $k$ spacers on top.  Cut $\mathcal{C}_2$ 
into $k$ subcolumns of equal width and place a single spacer on top 
of each subcolumn, then stack from left to right. 
After this stage, there are two columns of height $k ( h + 1 )$.  

\subsection{Independent cutting and stacking}
\label{ind-cut-stack-section}
Independent cutting and stacking is defined similar to \cite{shields1973cutting}.  
As opposed to \cite{shields1973cutting}, 
here it is not necessary to use columns of different heights, since weak mixing 
is generic and we are establishing a generic class of transformations.  
Also, in section \ref{rigid-inf-slow-ent-0}, we include a weak mixing stage 
which allows all columns to have the same height and facilitates counting of codewords.  
Given two columns $\mathcal{C}_1$ and $\mathcal{C}_2$ of height $h$, and $s \in \natural$, 
independent cutting and stacking the columns $s$ times produces $2^{2^{s}}$ columns, 
each with height $2^{s}h$.  

\subsection{Infinite upper slow entropy}
Let $P = \{ p_1, p_2 \}$ be a nontrivial measurable 2-set partition.  
We construct a dense $G_{\delta}$ for the case where $\mu (p_1) = \frac{1}{2}$, 
although a similar procedure will handle the more general case where $0 < \mu(p_1) < 1$.  
Let $a_n(t)$ be a sequence of real numbers with subexponential growth.  
In particular, for every $t, \beta > 1$, $\lim_{n\to \infty} \frac{a_n(t)}{\beta^n} = 0$.  
For $M,N,t \in \natural$, define 
\begin{align}
\overline{G}(M,N,t,P) = \{ T \in \mathcal{M} : \exists n > N\ \mbox{and}\ \delta > \frac{1}{M}
\ \mbox{such that}\ S_{P}(T, n, \delta, \delta) > a_n(t) \} . 
\end{align}

\begin{prop}
\label{open-infinity-prop}
For $M,N,t \in \natural$, the set $\overline{G}(M,N,t,P)$ is open in the weak topology on $\mathcal{M}$. 
\end{prop}
\begin{proof}
Let $T_0 \in \overline{G}(M,N,t,P)$ and $n>N$, $\delta_0 > \frac{1}{M}$ be such that 
$S_{P}(T, n, \delta_0, \delta_0) > a_n(t)$. 
Let $P_n = \bigvee_{i=0}^{n-1} T_0^{-i} P$.  The elements of $P_n$ of positive measure 
correspond to the various $P$-names of length $n$.  For almost every $x,y \in X$, 
$x$ and $y$ have the same $P$-name under $T_0$, if and only if 
$x, y \in p$ for some $p\in P_n$.  
Choose $\delta_1 \in \real$ such that $\frac{1}{M} < \delta_1 < \delta_0$.  
Let $\alpha = {\big( \delta_0 - \delta_1 \big)}/{2}$.  
In the weak topology, 
choose an open set $U$ containing $T_0$ such that for $T_1 \in U$, 
$0 \leq i < n$, and $p \in P_n$, 
\begin{align}
\mu ( T_1 T_0^i p \triangle T_0^{i+1}p ) &\leq \big( \frac{\alpha}{n^2} \big) \mu ( p ) \label{eqn3} .
\end{align}
We will prove inductively in $j$ for $p \in P_n$ that 
\begin{align}
\mu ( T_0^{j} p \triangle T_1^{j} p ) &\leq \big( \frac{j\alpha}{n^2} \big) \mu ( p ) . \label{eqn4}
\end{align}
The case $j=1$ follows directly from (\ref{eqn3}): 
\begin{align*}
\mu ( T_0 p \triangle T_1 p ) &\leq \big( \frac{\alpha}{n^2} \big) \mu (p) . 
\end{align*}
Also, the case $j=0$ is trivial. 
Suppose equation (\ref{eqn4}) holds for $j = i$.  Below shows it holds for $j = i+1$: 
\begin{align*}
\mu ( T_0^{i+1} p \triangle T_1^{i+1} p ) &\leq \mu ( T_0^{i+1} p \triangle T_1 T_0^{i} p ) + 
\mu ( T_1 T_0^{i} p \triangle T_1^{i+1} p ) \\ 
&= \mu ( T_0^{i+1} p \triangle T_1 T_0^{i} p ) + \mu ( T_0^{i} p \triangle T_1^{i} p ) \\ 
&\leq  \big( \frac{\alpha}{n^2} \big) \mu ( p ) + \big( \frac{i \alpha}{n^2} \big) \mu ( p ) \\ 
&= \big( \frac{(i+1)\alpha}{n^2} \big) \mu ( p ) . 
\end{align*}
For $p \in P_n$, let 
\[
E_p = \bigcap_{i=0}^{n-1} T_1^{-i} T_0^{i} p . 
\]
Thus, 
\begin{align}
\mu \big( E_p \big) &\geq \mu (p) - \sum_{i=1}^{n-1} \mu (p \triangle T_1^{-i}T_0^{i} p) \\ 
&= \mu (p) - \sum_{i=1}^{n-1} \mu (T_1^{i} p \triangle T_0^{i} p) \\
&> (1 - \alpha) \mu (p) . 
\end{align}
Hence, if $E = \bigcup_{p \in P_n} E_p$, $\mu(E) > 1 - \alpha$.  
Each $x \in E$ has the same $P$-name under $T_1$ and $T_0$.  
Suppose $V\subseteq 2^{\{0,1\}^n}$ is such that $A_1 = \bigcup_{v\in V} B^{T_1}_{\delta_1}(v)$ 
satisfies $\mu ( A_1 ) \geq 1 - \delta_1$. 
Let $A_0 = \bigcup_{v\in V} B^{T_0}_{\delta_0}(v)$ and 
$A_0^{\prime} = \bigcup_{v\in V} B^{T_0}_{\delta_1}(v)$.  
Since $\mu ( A_1 \triangle A_0^{\prime} ) \leq \mu(E^c) < \alpha$, then 
\[
\mu (A_0) \geq \mu (A_0^{\prime}) > 1 - \delta_1 - \alpha > 1 - \delta_0 . 
\]
Therefore, since $\mu (A_0) \geq 1 - \delta_0$, then $\mbox{card}(V) > a_n(t)$ and we are done. 
\end{proof}

Our density result follows. 
\begin{prop}
\label{dense-infinity-prop}
For sufficiently large $M$, and all $N,t \in \natural$, 
the set $\overline{G}(M,N,t,P)$ is dense in the weak topology on $\mathcal{M}$. 
\end{prop}
\begin{proof}
It will be sufficient to consider $M \geq 1000$.  
This will allow us to choose $\delta < \frac{1}{900}$.  Also, 
$d$ in Lemma \ref{ham2} can be chosen $d \geq \frac{1}{100}$.  
The value $m$ in Lemma \ref{ham2} will equal $2^r$ for some $r$.  
It will not be difficult to choose $r$ such that $2^r > 100$.  
Also, let $0 < \eta < 10^{-5}$.  
Thus, 
\[
\mathcal{H} \big( \frac{2(\delta + \eta)}{d} + 2^{-r} \big) < 
\mathcal{H} \big( \frac{2}{9} + \frac{1}{500} + \frac{1}{100} \big) 
< \mathcal{H} \big( \frac{1}{4} \big) < \frac{7}{8} . 
\]

Since rank-ones are dense in $\mathcal{M}$, let $T_0 \in \mathcal{M}$ be a rank-one transformation. 
Let $\epsilon > 0$.  We will show there exists $T_1$ within $\epsilon$ of $T_0$ 
in the weak topology.  It is sufficient to construct $T_1$ such that 
$\mu ( \{ x : T_1(x) \neq T_0(x) \} ) < \epsilon$.  We can reset $\epsilon = \min{\{ \epsilon, \eta \}}$. 
Let $N,t \in \natural$ and $p=p_1$.  Since $T_0$ is ergodic, we can choose 
$K\in \natural$ such that for $k \geq K$, 
\begin{align}
\int_X | \frac{1}{k} \sum_{i=0}^{k-1} I_{p}(T_0^i x) - \mu (p) | d\mu &< \frac{\eta^2}{6} . 
\label{eqn5}
\end{align}
Choose a rank-one column $\mathcal{C} = \{ I_0, I_1, \ldots , I_h \}$ 
for $T_0$ such that 
\begin{enumerate}
\item $\mu ( \bigcup_{i=0}^{h-2} I_i ) > 1 - \epsilon$, 
\item $h > N$, 
\item there exists a collection $J$ such that 
$\mu ( p \triangle \bigcup_{j\in J} I_j ) < \frac{\eta^2}{6} \mu (p)$.  
\end{enumerate}
Let $q = \bigcup_{j\in J} I_j$ and $Q = \{ q, q^c \}$.  Since $\mu(p) = {1}/{2}$, 
we can assume ${1}/{3} < \mu(q) < {2}/{3}$.  
Thus,
\begin{align*}
\int_X | \frac{1}{k} \sum_{i=0}^{k-1} I_q ( T_0^i x ) - \mu (q) | d\mu &\leq 
\int_X | \frac{1}{k} \sum_{i=0}^{k-1} \big( I_q ( T_0^i x ) - I_p(T_0^i x) \big) | d\mu \\ 
&+ \int_X | \frac{1}{k} \sum_{i=0}^{k-1} I_p ( T_0^i x ) - \mu (p) | d\mu +
| \mu (p) - \mu (q) | \\ 
&< \mu ( q \triangle p ) + \frac{\eta^2}{6} + \frac{\eta^2}{6} < \frac{\eta^2}{2} . 
\end{align*}

The transformation will apply independent cutting and stacking using 
two words who have a significant distance.  The words are pure with respect 
to a subset of intervals $I_j, j\in J$.  Since we are counting balls using 
the partition $P$, we will apply Lemma \ref{ham2} to covers of $P$-names. 

Now we show how to construct a transformation $T_1 \in \overline{G}(M,N,t,P)$ such that 
$d_w(T_0,T_1) < \epsilon$. 
Cut column $\mathcal{C}_1 = \{ I_0, I_1, \ldots , I_{h-1} \}$ into 2 subcolumns of equal width, 
labeled as $\mathcal{C}_{1,0}$ and $\mathcal{C}_{1,1}$.  For $k = 2(K+2)$ as above, cut each subcolumn 
into $k$ subcolumns of equal width.  For $\mathcal{C}_{1,0}$, stack from left to right, 
and then place $k$ spacers on top.  For $\mathcal{C}_{1,1}$, place a spacer on top, and then 
stack from left to right.  The auxillary level $I_h$ may be used to add these 
spacers.  Call these columns $\mathcal{C}_{2,0}$ and $\mathcal{C}_{2,1}$ respectively.  

Before proceeding with the construction, let us demonstrate the usefulness 
of these two blocks.  Suppose a copy of $\mathcal{C}_{2,0}$ is shifted by $j$ and 
we measure the overlap with two catenated unshifted copies of $\mathcal{C}_{2,1}$.  
The shifted copy of $\mathcal{C}_{2,0}$ will have an overlap of at least 
$(K+2)$ copies of $\mathcal{C}_{1,0}$ with one of the copies of $\mathcal{C}_{2,1}$.  
If $j \leq (K+2)(h+1)$, then consider the overlap with the first copy, 
otherwise consider the larger overlap with the second copy. 
Assume $j \leq (K+2)(h+1)$.  The other case is handled similarly. 
For $h$ sufficiently large, there will be an overlap of at least 
$K$ full copies of $\mathcal{C}_{1,0}$ with a copy of $\mathcal{C}_{1,1}$ 
which we call $\mathcal{C}^{\prime}$.  
Also, the copies will be distributed equally with shifts 
$j^{\prime}, j^{\prime}+1, j^{\prime}+2, \ldots , j^{\prime}+K-1$ for some $j^{\prime}$.   
Since $T_1(x) = T_0(x)$ for $x \in \mathcal{C}^{\prime}$, if $\alpha = \mu(q)$, 
\begin{align}
\mu \big( T_1^j q \cap q^c \cap \mathcal{C}^{\prime} \big) 
&\geq \frac{\mu(\mathcal{C}^{\prime})}{k} \sum_{i=0}^{K-1} \mu ( T_1^{j+i}q \cap q^c ) \\ 
&= \frac{\mu(\mathcal{C}^{\prime})K}{k} \frac{1}{K} \sum_{i=0}^{K-1} \mu ( T_1^{j+i}q \cap q^c ) \\ 
&\geq \frac{\mu(\mathcal{C}^{\prime})}{6} \Big( \frac{1}{K} 
\sum_{i=0}^{K-1} \mu ( T_1^{j+i}q \cap q^c ) \Big) \\ 
&\geq \frac{\mu(\mathcal{C}^{\prime})}{6} \Big( \frac{1}{K} \sum_{i=0}^{K-1} 
\mu ( T_1^{j+i}q \cap q^c ) - \alpha(1-\alpha) \Big) + 
\frac{\mu(\mathcal{C}^{\prime})\alpha(1-\alpha)}{6} \\ 
&> \mu(\mathcal{C}^{\prime}) \Big( \frac{\alpha(1-\alpha)}{6} - \frac{\eta^2}{6} \Big) 
> \mu(\mathcal{C}^{\prime}) \Big( \frac{1}{54} - \frac{\eta^2}{6} \Big) 
> \frac{\mu(\mathcal{C}^{\prime})}{100} .  
\end{align}

Let $\ell = k(h+1)$ and $\beta = 2^{\frac{1}{8k(h+1)}}$.  
Choose $r \in \natural$ such that for $n = 2^r k(h+1)$, 
\[
\beta^{n} > 2 a_{n} ( t ) . 
\]
Choose $s \in \natural$, $s > r$ such that $2^{r-s} < \eta^2$.  

By inequality (\ref{eqn5}), distance between 
shifts of the $P$-word formed from $\mathcal{C}_{1,1}$ and the $P$-word formed from $\mathcal{C}_{1,0}$ 
will be bounded away from 0.  Both columns have length $k(h+1)$.  
Remark: we cut into $k$, so that any two shifted blocks will have at least 
$K$ sub-blocks that overlap. 

Apply independent cutting and stacking to both columns, $s$ number of times. 
This will cause the number of $P$-names of length $k(h+1)$ to grow exponentially in $s$. 
(Note, a Chacon-2 similar stage could be inserted to guarantee weak mixing, 
but this is not needed, since weak mixing is generic.)  
There will be $2^{2^s}$ columns, each of height $2^s k(h+1)$.  
We will consider $P$-names of length $n$ where $n$ is large compared to $k(h+1)$, 
but small relative to $2^s k(h+1)$.  
For most points $y\in X$, we get a $P$-name by taking $x$ in the base of a column 
and some $j$ such that $y = T_1^j x$.  We can get a $P$-name for $y$ by forming 
the vector $v = \langle I_{p}(T^{j+i} x) \rangle_{i=0}^{n-1}$.  
Let $b$ be the set containing the bases of all columns of height $2^s k(h+1)$.  
Define $b_j = T_1^{j} b$ and 
$\vec{b_j} = \{ \langle I_{p}(T_1^{j+i} x) \rangle_{i=0}^{n-1} | x\in b, 0\leq i < n \}$.  
Suppose $S = S_P(T_1,n,\delta,\delta)$ is the minimal number of $\delta$-balls such that 
the union of measure is at least $1-\delta$.  
Let $v_1, v_2, \ldots , v_S$ be codewords of length $n$ such that 
$B_{\delta}^{T_1,P}(v_i)$ cover at least measure $1 - \delta$.  
Let $B = \bigcup_{i=1}^{S} B_{\delta}^{T_1,P}(v_i)$.  
Thus, 
\begin{align}
\mu \Big( B \cap \big( \bigcup_{i=0}^{n-1} \bigcup_{j=0}^{2^{s-r}-2} b_{jn+i_0} \big) \Big) 
&> 1 - \delta - \epsilon - 2^{r-s} . 
\end{align}
There exists $i_0$ such that 
\begin{align}
\mu \Big( B \cap \big( \bigcup_{j=0}^{2^{s-r}-2} b_{jn+i_0} \big) \Big) 
&> \frac{1}{n} \big( 1 - \delta - \epsilon - 2^{r-s} \big) 
\end{align}
Let $D = \bigcup_{j=0}^{2^{s-r}-2} b_{jn+i_0}$.  Then 
\begin{align}
\mu \Big( B \cap D \Big) 
&> \frac{\mu(D)}{n\mu(D)} \big( 1 - \delta - \epsilon - 2^{r-s} \big) \\ 
&\geq \frac{(1-\delta-\epsilon-2^{r-s})}{(1-\epsilon-2^{r-s})} \mu(D) \\ 
&= \big( 1 - \frac{\delta}{1-\epsilon-2^{r-s}} \big) \mu (D) . 
\end{align}
Hence, 
\[
\mu ( D \setminus B ) < \big( \frac{\delta}{1-\epsilon-2^{r-s}} \big) \mu (D). 
\]
Let 
\[
J_1 = \{ j : 0\leq j \leq 2^{s-r}-2, \mu ( b_{jn+i_0} \setminus B ) < 
\frac{2\delta}{1 - \epsilon - 2^{r-s}} \mu ( b_{jn+i_0}) \} . 
\]
Let $\hat{b}_{j} = \bigcup_{i=0}^{n-1} b_j$.  
Note 
\[
\sum_{j=0}^{2^{s-r}-2} \mu ( q \cap \cup_{j=0}^{2^{s-r}-2} \hat{b}_{jn+i_0} 
\triangle p \cap \cup_{j=0}^{2^{s-r}-2} \hat{b}_{jn+i_0} ) < \frac{\eta^2}{6} . 
\]
Let 
\[
J_2 = \{ j : 0\leq j \leq 2^{s-r}-2, \mu ( q \cap \hat{b}_{jn+i_0} 
\triangle p \cap \hat{b}_{jn+i_0} ) < \frac{\eta^2}{3(2^{s-r}-1)} \} . 
\]
Both $|J_1| > \frac{1}{2} \big( 2^{s-r} - 1 \big)$ and 
$|J_2| > \frac{1}{2} \big( 2^{s-r} - 1 \big)$.  
Thus, there exists $j_0 \in J_1 \cap J_2$.  

Since $\mu ( q \cap \hat{b}_{j_0n+i_0} \triangle p \cap \hat{b}_{jn+i_0} ) 
< \eta^2 \mu ( \hat{b}_{j_0 n + i_0} )$, by Lemma \ref{ham4}, 
\[
\mu \big( \{ x \in b_{j_0n+i_0} : d( \vec{Q}_n(x) , \vec{P}_n(x) ) < \eta \} \big) > 1 - \eta . 
\]
Therefore, by Lemma \ref{ham2}, the number of Hamming balls $S$ satisfies: 
\begin{align}
S \geq & \big( 1 - \frac{2\delta}{1 - \epsilon - 2^{r-s}} - \eta \big) 
2^{ 2^r (1 - \mathcal{H}( \frac{ 2( \delta + \eta ) }{ d } + 2^{-r} )) } \\ 
\geq &  \big( 1 - \frac{2\delta}{1 - \epsilon - 2^{r-s}} - \eta \big) \beta^{ n } \\ 
> & 2 \big( 1 - \frac{2\delta}{1 - \epsilon - 2^{r-s}} - \eta \big) a_n(t) 
> a_n(t) . 
\end{align}


Therefore, $S_{P}(T_1,n,\delta,\delta) > a_n(t)$ and $T_1 \in \overline{G}(M,N,t,P)$. 
\end{proof}

\noindent 
Remark: It is clear from the proof of Proposition \ref{dense-infinity-prop} that given 
$\gamma > 0$, there exists $M(\gamma)$ such that $\overline{G}(M,N,t,P)$ 
is dense for $M\geq M(\gamma)$ and 2-set partition $P$ such that $\mathcal{H}(P) \geq \gamma$.  

 \begin{thm}\label{residual-inf-thm}
 Suppose $a_n(t)$ is such that for $t,\beta > 1$, $\lim_{n\to \infty} \frac{a_n(t)}{\beta^n} = 0$.  
 There exists a dense $G_{\delta}$ subset $G \subset \mathcal{M}$ such that for each 
 $T \in G$, the upper slow entropy of $T$ is infinite with respect to $a_n(t)$. 
 \end{thm}
 \begin{proof}
 By Proposition \ref{dense-infinity-prop}, for $M$ sufficiently large, $\overline{G}(M,N,t,P)$ is dense 
 for all $N, t \in \natural$.  Also, $\overline{G}(M,N,t,P)$ is open 
by Proposition \ref{open-infinity-prop}.  
 Thus, 
 \[
 G = \bigcap_{t=1}^{\infty} \bigcap_{N=1}^{\infty} \overline{G}(M,N,t,P) 
 \]
 is a dense $G_{\delta}$.  Fix $t \in \natural$.  
 For $T \in \bigcap_{N=1}^{\infty} \overline{G}(M,N,t,P)$, 
 \[
 \limsup_{n\to \infty} \frac{ S_P(T,n,\frac{1}{M},\frac{1}{M}) }{a_n(t)} \geq 1 . 
 \]
 Therefore, the upper slow entropy is $\infty$ with respect to $a_n(t)$. 
\end{proof}

\begin{cor}
In the weak topology, the generic transformation in $\mathcal{M}$ is rigid, weak mixing, 
rank-one and has infinite polynomial upper slow entropy. 
\end{cor}

The following corollary is a strengthening of Theorem \ref{residual-inf-thm}.  
\begin{cor}
\label{gen-cor}
Given a rate function $b = b_n(t)$, let 
\begin{align}
\mathcal{G}_b = \{ T \in \mathcal{M}:  \mbox{s-} \overline{\mathcal{H}}_{b}^{\mu} (T,P) > 0 
\ \mbox{for every nontrivial finite measurable}\ P \} . \label{slow-res-class}
\end{align}
If $b_n(t)$ is subexponential, then $\mathcal{G}_b$ contains a dense $G_{\delta}$ class 
in $\mathcal{M}$.  
\end{cor}
\begin{proof}
It is sufficient to prove this corollary for 2-set partitions.  
Let $P_L$ be a sequence of nontrivial measurable 2-set partitions such that 
the collection $\{ P_L : L \in \natural \}$ is dense in the class of all 
measurable 2-set partitions.  
By the proofs of Propositions \ref{open-infinity-prop} and \ref{dense-infinity-prop}, 
then for $N,t,L \in \natural$ and $M_L$ sufficiently large,  
the set $G(M,N,t,P_L)$ is open and dense in the weak topology 
for $M \geq M_L$.  
Thus, 
\[
G = \bigcap_{L=1}^{\infty} \bigcap_{M=M_L}^{\infty} \bigcap_{t=1}^{\infty} \bigcap_{N=1}^{\infty} 
\overline{G}(M,N,t,P_L) 
\]
is a dense $G_{\delta}$.  
Also, $M_L$ may be set equal to $M(\gamma)$ where 
$\gamma = \mathcal{H}(P_L)$.  Actually, choose $M \geq M(\frac{\gamma}{2})$ 
such that if 
\[
D(P,Q) < \frac{1}{9M^2}, 
\]
then $\mathcal{H}(Q) > \mathcal{H}(P) - \frac{\gamma}{2}$.  
Choose $L_0 \in \natural$ such that $D(P,P_{L_0}) < \frac{1}{9M^2}$.  
Let $t > 0$.  
For $T \in \bigcap_{N=1}^{\infty} \overline{G}(M,N,t,P_{L_0})$, 
\[
\limsup_{n\to \infty} \frac{ S_{P_{L_0}}(T,n,\frac{1}{M},\frac{1}{M}) }{b_n(t)} \geq 1 . 
\]
By Lemma \ref{ham5}, 
$S_{P}(T,n,\frac{1}{3M},\frac{1}{3M}) \geq S_{P_{L_0}}(T,n,\frac{1}{M},\frac{1}{M})$.  
Therefore, for $T\in G$, 
$\mbox{s-} \overline{\mathcal{H}}_{b}^{\mu} (T,P) > 0$ 
which implies $G \subseteq \mathcal{G}_b$.  
\end{proof}

\section{Infinite lower slow entropy rigid transformations}
\label{rigid-inf-slow-ent-0}
All transformations are constructed on $[0,1]$ with Lebesgue measure.  
The transformations with infinite lower slow entropy will be infinite rank transformations 
defined using three different types of stages.  These stages will be repeated 
in sequence ad infinitum.  If all stages are labeled $\mathcal{S}_i$ for $i=0,1,2,\ldots$, then 
the stages break down as follows: 
\begin{enumerate}
\item $\mathcal{S}_{3i}$ will be an independent cutting and stacking stage; 
\item $\mathcal{S}_{3i+1}$ will be a weak mixing stage (similar to Chacon-2); 
\item $\mathcal{S}_{3i+2}$ will be a rigidity stage. 
\end{enumerate} 
The transformation will be initialized with two columns $\mathcal{C}_{0,1}$ and $\mathcal{C}_{0,2}$ 
each containing a single interval. Let $P$ be the 2-set partition containing 
each of $\mathcal{C}_{0,1}$ and $\mathcal{C}_{0,2}$.  As the transformation is defined, we will add 
spacers infinitely often.  Spacers will be unioned with the second set, so 
in general, the partition $P = \{ C_{0,1}, X\setminus C_{0,1} \}$. 

\subsection{Independent cutting and stacking stage}
\label{ics}
To define the independent cutting and stacking stage, we use a sequence $s_k \in \natural$ 
for $n=0,1,\ldots$ such that $s_k \to \infty$.  
For stage $\mathcal{S}_0$, independent cut and stack $s_0$ times beginning with columns 
$\mathcal{C}_{0,1}$ and $\mathcal{C}_{0,2}$.  The number of columns is squared at each cut and stack stage, 
so after $s_0$ substages, there are $2^{2^{s_0}}$ columns each having height $2^{s_0}$. 

For the general stage $\mathcal{S}_{3k}$, suppose 
there are $\ell$ columns each having height $h$ at the start of the stage. 
If we independent cut and stack $s_k$ times, then there will be 
\[
2^{2^{s_k}} \ell 
\]
columns each having height $2^{s_k} h$. 

\subsection{Weak mixing stage}
\label{wkmix-stage}
For the general weak mixing stage $\mathcal{S}_{3k+1}$, suppose there are $\ell$ columns of height $h$.  
Cut each column into 2 subcolumns of equal width, add a single spacer on the right subcolumn, 
and stack the right subcolumn on top of the left subcolumn.  
Thus, there will be $\ell$ columns of height $2h+1$. 

Suppose $f$ is an eigenfunction with eigenvalue $\lambda$.  Thus, $f(Tx) = \lambda f(x)$. 
There will be a weak mixing stage with refined columns such that $f$ is 
nearly constant on most intervals.  Let $I$ be one such interval.  
Suppose $x, y \in I$ such that 
$T^h x \in I$ and $T^{h+1}y \in I$ and $f$ is nearly constant on each of these four values.  
Since the following stage is a rigidity stage, it is not difficult to show 
points $x$ and $y$ exist with this property. 
Then 
\[
f( T^h x ) = \lambda^h f(x) \approx f( T^{h+1} y ) = \lambda^{h+1} f ( y ) . 
\]
Since $f(x) \approx f(y)$, then $\lambda^h \approx \lambda^{h+1}$. 
Hence, $\lambda \approx 1$.  In the limit, this shows that $\lambda = 1$ is the only eigenvalue, 
and since $T$ will be ergodic, $f$ is constant and $T$ is weak mixing. 
This is the same argument used in the original proof of Chacon that 
the Chacon-2 transformation is weak mixing \cite{Chacon}. 

\subsection{Rigidity stage}
\label{rigid-stage}
A sequence $r_k \in \natural$ is used to control rigidity.  Suppose at stage 
$\mathcal{S}_{3k+2}$, there are $\ell$ columns of height $h$.  Cut each column into $r_k$ 
subcolumns of equal width, and stack from left to right to obtain $\ell$ columns 
of height $r_k h$. 
Thus, for any set $A$ that is a union of intervals from the $\ell$ columms, 
then
\[
\mu ( T^h A \cap A ) > \Big( \frac{r_k - 1}{r_k} \Big) \mu (A) . 
\]

\subsection{Infinite lower slow entropy} \label{rigid-inf-slow-ent}
Our family $\mathcal{F}$ of transformations are parameterized by sequences 
$r_k \to \infty$ and $s_k \in \natural$.  
Initialize $T \in \mathcal{F}$ using two columns of height 1. 
Independent cut and stack $s_0$ times.  This produces 
$2^{2^{s_0}}$ columns of height $2^{s_0}$.  
Let $h_1 = 2^{s_0}$ be the heights of these columns. 
Cut each column into two subcolumns of equal width, add a single spacer 
on the rightmost subcolumn and stack the right subcolumn on the left subcolumn. 
Then cut each of these columns of height $2h_1 + 1$ into $r_1$ subcolumns 
of equal width and stack from left to right. 
Independent cut and stack these $2^{2^{s_0}}$ columns of height 
$r_1 ( 2h_1 + 1 )$ to form $2^{2^{s_0+s_1}}$ columns of height 
$h_2 = 2^{s_1} r_1 ( 2h_1 + 1 )$.  
In general, $h_k = 2^{s_{k-1}} r_{k-1} ( 2h_{k-1} + 1 )$.  
Also, let 
\[
\sigma_k = \sum_{i=0}^{k-1} s_i . 
\]
Remark: As long as $s_k > 0$ infinitely often, 
the transformations are ergodic due to the independent 
cutting and stacking stages.  Thus, based on section \ref{wkmix-stage}, 
the transformations will be weak mixing.  Rigidity will follow from a given 
sequence $r_k \to \infty$.  The proof that the lower slow entropy is infinite is given next. 

\begin{thm}
Given any subexponential rate $a_n(t)$, 
there exists a rigid weak mixing system $T \in \mathcal{F}$ 
such that $T$ has infinite lower slow entropy with respect to $a_n(t)$. 
\end{thm}
\begin{proof}
Let $\varepsilon \in \real$ such that $0 < \varepsilon < \frac{1}{100}$.  
Suppose $r_k, t_k \in \natural$ such that 
$\lim_{k\to \infty} r_k = \lim_{k\to \infty} t_k =\infty$.  
Let $h_0 = 1$ and $h_1 = 2^{s_0}$.  Recall the formulas for $h_{k+1}$ 
and $\sigma_k$, 
\[
h_{k+1} = 2^{s_k} r_k \big( 2h_k + 1 \big) \ \ \mbox{and}\ \ \sigma_k = \sum_{i=0}^{k-1} s_i . 
\]
We will specify the sequence $s_k$ inductively 
based on $\sigma_{k}, r_k, h_k$ and $r_{k+1}$.   
For sufficiently large $k$, 
\[
2^{2^{\sigma_k} \mathcal{H} ( 2\varepsilon + 2^{-\sigma_k} ) } 
<  2^{2^{\sigma_k} \mathcal{H} ( 3\varepsilon ) } . 
\]
Let 
\[
\alpha_k = \frac{2^{\sigma_{k}} \big( 1-\mathcal{H}(3\varepsilon) \big)}
{2^5 r_{k+1} r_k h_k}  \ \ \mbox{and}\ \ \beta_k = 2^{\alpha_k} . 
\]
Choose $s_k \in \natural$ such that for $n \geq 2^{s_k}$, 
\[
\beta_k^n > k a_n ( t_k ) . 
\]

Let $0 < \delta < \frac{1}{100}$.  Suppose $S \in \natural$ 
and $v_i$, $1 \leq i \leq S$, are such that 
\[
\mu \big( \bigcup_{i=1}^{S} B_{\varepsilon}^{T,P} (v_i) \big) > 1-\delta . 
\]
Suppose a large $n \in \natural$ is chosen.  
There exists $k \in \natural$ such that $2h_k \leq n < 2h_{k+1}$. 
We divide the proof into two cases: 
\begin{enumerate}
\item $2 r_k ( 2h_k + 1 ) \leq n < 2^{s_k+1} r_k ( 2h_k + 1 )$ 
\item $2 h_k \leq n < 2 r_k ( 2h_k + 1 )$ . 
\end{enumerate}
For case (1), 
there exist $\rho \in \natural$ such that $1 \leq \rho \leq s_k$ and 
\[
2^{\rho} r_k ( 2h_k + 1 ) \leq n < 2^{\rho + 1} r_k ( 2h_k + 1 ) . 
\]
Our $T$ is represented as a cutting and stacking construction 
with the number of columns tending to infinity. 
Fix a height $H$ much larger than $n$ and set of columns 
$\mathcal{C}$.  
Pick $H$ such that $\frac{n}{H} < \frac{\eta}{2}$ and 
the complement of columns in $\mathcal{C}$ is less that 
$\frac{\eta}{2}$ for small $\eta$.  
Let $b$ be points in the bottom levels of $\mathcal{C}$.  
Also, let $b_j = T^j b$ for $0\leq j < H - n$.  

Let $B = \bigcup_{i=1}^{S} B_{\varepsilon}^{T,P} (v_i)$.  There exists 
$j < H - n$ such that 
\[
\mu ( B \cap b_j ) > \big( \frac{1 - \delta - \eta}{1 - \eta} \big) 
\mu ( b_j ) 
> ( 1 - 2\delta ) \mu(b_j) \ \ \mbox{for sufficiently small $\eta$}. 
\]
If $\mu ( B_{\varepsilon}^{T,P} (v_i) \cap b_j ) = 0$ for some $i$, 
then remove that Hamming ball from the list.  
Let $R \leq S$ be the number of remaining Hamming balls and rename 
the vectors so that 
$B = \bigcup_{i=1}^{R} B_{\varepsilon}^{T,P} (v_i)$ satisfies 
$\mu(B\cap b_j) > ( 1 - 2\delta ) \mu (b_j)$.  
Choose $x_i \in B_{\varepsilon}^{T,P} (v_i) \cap b_j$ such that 
$\mu ( \{ x \in B \cap b_j : \vec{P}_n(x) = \vec{P}_n(x_i) \} ) > 0$.  
Let $w_i = \vec{P}_n(x_i)$.  Thus, by the triangle inequality 
of the Hamming distance, 
$B_{2\varepsilon}( w_i ) \supseteq B_{\varepsilon} ( v_i )$ and hence, 
\[
B_{2\varepsilon}^{T,P}( \vec{P}_n(x_i) ) \cap b_j 
\supseteq B_{\varepsilon}^{T,P} (v_i) \cap b_j . 
\]
This previous statement allows us to consider only balls 
with $P$-names as centers by doubling the radius. 
For a.e. point $x \in b_j$, the $h_k$ long blocks, 
called $\mathcal{C}_k$, will align.  Also, lower level blocks 
$h_i$ for $i < k$ will align as well. 

Since all $h_k$ blocks and sub-blocks align for $x \in b_j$, 
we can view this problem as a standard estimate on hamming ball 
sizes for i.i.d. binary sequences.  
The spacers added over time diminish and will have little effect 
on the distributions and hamming distance.  
For repeating blocks to be within $2\varepsilon$ of a word $w_i$, 
it is necessary the sub-blocks to be within approximately 
$2\varepsilon$.  
Note, 
\[
n \geq 2^{\rho} r_k ( 2h_k + 1 ) 
= 2^{\rho} r_k \Big( 2 
\big( 2^{s_{k-1}} r_{k-1} (2h_{k-1} + 1) \big) \Big) > 2^{s_{k-1}} . 
\]
Since we have approximated must but not all variables, 
we can make $R$ at least, 
\[
R \geq \frac{1}{2} \big( 1 - 2\delta \big) 
2^{ 2^{\rho} 2^{\sigma_k} \big( 1 - \mathcal{H}( 3\varepsilon ) \big) } . 
\]
Hence,
\begin{align}
2^{ 2^{\rho} 2^{\sigma_k} \big( 1 - \mathcal{H}( 3\varepsilon ) \big) } 
&= 2^{ 2^{\rho + s_{k-1}} 2^{\sigma_{k-1}} \big( 1 - \mathcal{H}( 3\varepsilon ) \big) } \\ 
&= \Big( 2^{\frac{(1 - \mathcal{H}(3\varepsilon)) 2^{\sigma_{k-1}} }
{32r_k r_{k-1} h_{k-1}}} \Big)^{2^{\rho + s_{k-1}} 32 r_k r_{k-1} h_{k-1}} \\ 
&> \Big( 2^{\alpha_{k-1}} \Big)^{2^{\rho} 8 r_k h_{k}} \\ 
&> \big( 2^{\alpha_{k-1}} \big)^{ 2^{\rho+1} r_k (2h_k + 1)} \\ 
&> \big( \beta_{k-1} \big)^{n} > (k-1) a_n ( t_{k-1} ) . 
\end{align}
This handles case (1).  Case (2) can be handled in a similar fashion, 
with focus on the blocks of length $h_{k-1}$. 
\end{proof}

\section{Slow entropy and the convergence rates of Blume}
This section is devoted to comparing and contrasting the notion of slow entropy 
with the study of entropy convergence rates of Blume \cite{blume2012relation, 
blume2015realizing}.  
In \cite{blume2012relation}, 
Blume proves that the set of all transformations $T$ satisfying the property that
for all nontrivial finite measurable partitions $P$, the entropy 
$\mathcal{H}(\bigvee_{i=0}^{n-1} T^{-i}P)$ of the $n^{th}$ refinements 
converges in the limit superior faster to $\infty$ than a given sublinear rate $a_n$ 
is residual with respect to the weak topology.  

The entropy convergence rates studied by Blume are fundamentally different from 
the slow entropy formulation introduced by Katok and Thouvenot.  
In particular, given a sublinear rate $a_n \to \infty$ and subexponential rate $b_n(t)$, 
we give an outline of a cutting and stacking construction 
which has zero upper slow entropy with respect to $b_n(t)$, but still satisfies 
for every nontrivial finite measurable partition $P$: 
\[
\limsup_{n\to \infty} \frac{1}{a_n} \mathcal{H} 
\big( \bigvee_{i=0}^{n-1} T^{-i}P \big) = \infty . 
\]
We say $a=a_n$ is sublinear, if $\lim_{n\to \infty} a_n = \infty$ and 
$\lim_{n\to \infty} \frac{a_n}{n} = 0$.  
Previously introduced by Blume \cite{blume2012relation}, 
the class $ES(a)$ is defined as:
\[
ES(a) = \{ T \in \mathcal{M} : 
\limsup_{n\to \infty} \frac{1}{a_n} \mathcal{H} \big( \bigvee_{i=0}^{n-1} T^{-i}P \big) > 0 
\ \mbox{for all nontrivial finite}\ P \} . 
\]
Also, we show for any sublinear rate $a = a_n$, then for the subexponential rate 
$b = b_n(t) = 2^{ta_n}$, the following holds:
\footnote{The set $\mathcal{G}_b$ is defined in \ref{slow-res-class}.} 
\[
\mathcal{G}_b \subsetneq ES(a) . 
\]
This demonstrates that the slow entropy formulation provides a stricter measure 
of complexity for a dynamical system than the entropy convergence rate.


\subsection{Infinite entropy convergence rate and zero lower slow entropy}
The following result distinguishes slow entropy from entropy convergence rate. 

\begin{prop}
\label{example-prop}
Given a sublinear rate $a_n$ and any subexponential rate $b=b_n(t)$, 
there exists an ergodic invertible measure preserving transformation $T$ 
such that $T$ has zero upper slow entropy with respect to $b_n(t)$, 
but for any nontrivial finite measurable partition $P$, 
\[
\limsup_{n \to \infty} \frac{1}{a_n} \mathcal{H} 
\big( \bigvee_{i=0}^{n-1} T^{-i} P \big) > 0 . 
\]
\end{prop}
\begin{proof}
It is sufficient to prove this for any nontrivial measurable 2-set 
partition $P = \{ p, p^c \}$.  
First we describe the inductive step for constructing $T$.  
Let $\mathcal{C}_1$ and $\mathcal{C}_2$ be columns of height $h_n$.  
Cut $\mathcal{C}_1$ into $2(n+2)$ subcolumns of equal width and 
cut $\mathcal{C}_2$ into 2 subcolumns of equal width.  
Swap the last subcolumn of $\mathcal{C}_1$ with the second 
subcolumn of $\mathcal{C}_2$.  
For the modified $\mathcal{C}_2$, cut into 2 subcolumns of 
equal width.  Cut each of these subcolumns into $k_n$ subcolumns 
of equal width and use spacers to build nearly independent words 
as described in section \ref{two-ind-section}.  
Then apply independent cutting and stacking as described 
in section \ref{ind-cut-stack-section} to the two $\mathcal{C}_2$ subcolumns 
to create $2^{2^{s_n}}$ subcolumns of equal width. 
Each of these subcolumns will have height 
$2^{s_n} k_n (h_n + 1)$.  
Stack these subcolumns to build a single column 
of height $2^{2^{s_n}} 2^{s_n} k_n (h_n + 1)$.  

For the modified $\mathcal{C}_1$ column, cut into $k_n$ subcolumns 
of equal width, stack from left to right, and add $k_n$ spacers on top.  
Then cut this column into 2 subcolumns of equal width, stack the right subcolumn 
on top of the left subcolumn.  Repeat this procedure a total of $2^{2^{s_n}+s_n}$ times.  
After this stage, we have produced two subcolumns of equal height, 
\[
h_{n+1}  = 2^{2^{s_n}} 2^{s_n} k_n (h_n + 1) . 
\]

This process can be initialized by taking any rank-one transformation 
and dividing a column of height $h_1$ into 2 subcolumns of equal width. 
After the first stage, there will be 2 subcolumns, one with width 
$\frac{2}{3}$ and the other with width $\frac{1}{3}$.  
After $n-1$ stages, there will be 2 subcolumns of height $h_{n}$, 
one with width $\frac{n}{n+1}$ and the other with width $\frac{1}{n+1}$.  

For rapidly growing $s_n$, this transformation $T$ will boost the 
entropy on the right column for each stage.  
In particular, given sublinear $a_n$, it is possible to choose $s_n$ 
such that the right portion appears to be a positive entropy transformation, 
but scaled by a factor of $\frac{1}{n+1}$.  
More explicitly, if $X_n$ represents the left column at stage $n$, 
and $U_n \subset X_n$ is the portion that is switched, then 
$U_i$ will be approximately conditionally independent of $U_j$ on $X_n$.  
Thus, by the 2nd Borel-Cantelli lemma, a.e. point will fall in $U_i$ 
for infinitely many $i$.  Moreover, for large $s_n$, near pairwise 
independence of $U_i$ will lead to near pairwise independence 
between any set $p$ and most $U_i$.  In particular, we can choose 
$s_n$ such that given $\gamma > 0$, 
\[
\Gamma_p = \{ i \in \natural : | \mu ( p \cap U_i ) - \mu(p)\mu(U_i) | < \gamma \mu (U_i) \} , 
\]
satisfies 
\[
\sum_{i \in \Gamma_p} \mu ( U_i ) < \infty . 
\]
Since $\lim_{n\to \infty} \mu(X_n) = 1$, $T$ will be ergodic.  This 
will enable the creation of two nearly independent words similar to section \ref{two-ind-section}.  
Thus, independent cutting and stacking raises the slow entropy on this portion (right column) 
of the space.  However, using cut and half and stack on the left column lowers 
the slow entropy on most of the space.  Hence, given $\delta > 0$, 
once the right column mass $\frac{1}{n+1}$ falls below $\delta$, 
then it will be possible to use a sublinear number of $P$-names 
to cover the names produced by the left column.  
The scaled entropy calculation $\frac{1}{a_n} \mathcal{H}(\bigvee_{i=0}^{n-1} T^{-i} P)$ 
will include the rising entropy from the right column of mass $\frac{1}{n+1}$.  

If sublinear $a_n$ is specified and $b=b_n(t)$ such that $\lim_{n\to \infty} b_n(t) = \infty$ 
for all $t>0$, then $s_n$ can be chosen to increase rapidly such that 
for every nontrivial measurable 2-set partition $P$, 
\[
\mbox{s-} \overline{\mathcal{H}}_{b}^{\mu} (T,P) = 0 
\]
and 
\[
\limsup_{n\to \infty} \frac{1}{a_n} \mathcal{H} \big( \bigvee_{i=0}^{n-1} T^{-i} P \big) > 0 . 
\]
Therefore, $T \in ES(a)$, but $T$ has zero upper slow entropy with respect to $b_n(t)$.  
\newline \noindent 
Remark: $b_n(t)$ may be chosen with slow growth and is not necessarily $2^{ta_n}$.  
\end{proof}

\subsection{$\mathcal{G}_{2^{ta_n}} \subsetneq ES(a_n)$}
It is straightforward to show that the rate function $b_n(t) = 2^{ta_n}$ is subexponential if and only if 
$a_n$ is sublinear.  
\begin{prop}
Suppose $a_n$ is sublinear and $b_n(t) = 2^{ta_n}$.  If $T \in \mathcal{G}_b$, then 
for every nontrivial finite measurable partition $P$, 
\[
\limsup_{n\to \infty} \frac{1}{a_n} \mathcal{H} \big( \bigvee_{i=0}^{n-1} T^{-i} P \big) > 0 . 
\]
Thus, $\mathcal{G}_b \subsetneq ES(a)$.  
\end{prop}
\begin{proof}
We prove the contrapositive.  Suppose $P$ is a finite measurable partition such that 
\[
\lim_{n\to \infty} \frac{1}{a_n} \mathcal{H} \big( \bigvee_{i=0}^{n-1} 
T^{-i} P \big) = 0. 
\]
Let $P_n = \bigvee_{i=0}^{n-1} T^{-i} P$.  Let $V_n = \{ p\in P_n : 0 < \mu(p) < 2^{-ta_n} \}$ 
and $V_n^{\prime} = \{ p\in P_n : \mu(p) \geq 2^{-ta_n} \}$.  
Thus, for $n$ sufficiently large, 
\begin{align}
\mathcal{H}\big( P_n \big) = & 
- \sum_{p\in V_n} \mu(p) \log{\mu(p)} - \sum_{p\in V_n^{\prime}} \mu(p) \log{\mu(p)} \\ 
\geq & - \sum_{p\in V_n} \mu(p) \log{2^{-ta_n}} - \sum_{p\in V_n^{\prime}} 2^{-ta_n} \log{2^{-ta_n}} \\ 
= &\ ta_n \mu(V_n) + | V_n^{\prime} | ta_n 2^{-ta_n} . 
\end{align}
This implies 
\[
\mu(V_n) \leq \frac{1}{ta_n} \mathcal{H} \big( P_n \big) \to 0,\ \mbox{as}\ n\to \infty . 
\]
Hence, for $\delta > 0$ and $n$ sufficiently large, $\mu(V_n) < \delta$, and 
\[
\frac{S_P(T,n,\delta,\delta)}{2^{ta_n}} \leq \frac{|V_n^{\prime}|}{2^{ta_n}} 
\leq \frac{1}{ta_n} \mathcal{H}(P_n) \to 0,\ \mbox{as}\ n\to \infty . 
\]
Since this is true for all $t>0$ and $\delta > 0$, then 
$\mbox{s-} \overline{\mathcal{H}}_{b}^{\mu} (T,P) = 0$ and therefore $T \notin \mathcal{G}_b$. 
\end{proof}
The previous proposition combined with Corollary \ref{gen-cor} 
gives an extension of Blume's Theorem 4.8 from \cite{blume2012relation}.  
Also, counterexamples from Proposition \ref{example-prop} 
demonstrate this is a nontrivial extension, 
since there exist $T \in ES(a)$ such that $T \notin \mathcal{G}_b$ for corresponding 
$b_n(t) = 2^{ta_n}$. 
Using the technique given in \cite{adams2015tower}, 
given any rigidity sequence $\rho_n$ for an ergodic invertible measure preserving transformation, 
it is possible to construct 
an ergodic invertible measure preserving transformation $T \in ES(a)$ which is rigid on $\rho_n$. 
It is an open question what families of rigidity sequences are realizable 
for transformations with infinite or positive lower slow entropy with respect 
to a given subexponential rate.

\section*{Acknowledgements.} 
We wish to thank Adam Kanigowski, Daren Wei and Karl Petersen for their feedback.

\bibliographystyle{alpha}
\bibliography{slow-ent}

\newcommand{\NoopSort}[1]{}
\begin{bibdiv}
\begin{biblist}

\bib{adams2015tower}{article}{
      author={Adams, Terrence},
       title={Tower multiplexing and slow weak mixing},
        date={2015},
     journal={Colloquium Mathematicum},
      volume={1},
      number={138},
       pages={47\ndash 71},
}

\bib{blume2012relation}{article}{
      author={Blume, Frank},
       title={On the relation between entropy convergence rates and baire
  category},
        date={2012},
     journal={Mathematische Zeitschrift},
      volume={271},
      number={3-4},
       pages={723\ndash 750},
}

\bib{blume2015realizing}{article}{
      author={Blume, Frank},
       title={Realizing subexponential entropy growth rates by cutting and
  stacking},
        date={2015},
     journal={Discrete \& Continuous Dynamical Systems-B},
      volume={20},
      number={10},
       pages={3435},
}

\bib{Chacon}{article}{
      author={Chacon, R.~V.},
       title={Weakly mixing transformations which are not strongly mixing},
        date={1969},
     journal={Proc. Amer. Math. Soc.},
      volume={22},
       pages={559\ndash 562},
}

\bib{ferenczi1996rank}{article}{
      author={Ferenczi, S{\'e}bastien},
       title={Rank and symbolic complexity},
        date={1996},
     journal={Ergodic Theory and Dynamical Systems},
      volume={16},
      number={4},
       pages={663\ndash 682},
}

\bib{ferenczi1997measure}{article}{
      author={Ferenczi, S{\'e}bastien},
       title={Measure-theoretic complexity of ergodic systems},
        date={1997},
     journal={Israel Journal of Mathematics},
      volume={100},
      number={1},
       pages={189\ndash 207},
}

\bib{friedman1972mixing}{article}{
      author={Friedman, N.~A.},
      author={Ornstein, D.~S.},
       title={On mixing and partial mixing},
        date={1972},
     journal={Illinois Journal of Mathematics},
      volume={16},
      number={1},
       pages={61\ndash 68},
}

\bib{Fri92}{article}{
      author={Friedman, N.A.},
       title={Replication and stacking in ergodic theory},
        date={1992},
     journal={Amer. Math. Monthly},
      volume={99},
      number={1},
       pages={31\ndash 41},
}

\bib{Fr70}{book}{
      author={Friedman, Nathaniel~A.},
       title={Introduction to {E}rgodic {T}heory},
   publisher={Van Nostrand Reinhold Co., New York-Toronto, Ont.-London},
        date={1970},
        note={Van Nostrand Reinhold Mathematical Studies, No. 29},
      review={\MR{0435350 (55 \#8310)}},
}

\bib{hochman2012slow}{article}{
      author={Hochman, Michael},
       title={Slow entropy and differentiable models for infinite-measure
  preserving k actions},
        date={2012},
     journal={Ergodic theory and dynamical systems},
      volume={32},
      number={2},
       pages={653\ndash 674},
}

\bib{KKW2020}{article}{
      author={Kanigowski, A.},
      author={Katok, A.},
      author={Wei, D.},
       title={Survey on entropy-type invariants of sub-exponential growth in
  dynamical systems},
        date={2020},
     journal={arXiv:2004.04655},
       pages={1\ndash 47},
}

\bib{kanigowski2018slow}{article}{
      author={Kanigowski, Adam},
       title={Slow entropy for some smooth flows on surfaces},
        date={2018},
     journal={Israel Journal of Mathematics},
      volume={226},
      number={2},
       pages={535\ndash 577},
}

\bib{kanigowski2019slow}{article}{
      author={Kanigowski, Adam},
      author={Vinhage, Kurt},
      author={Wei, Daren},
       title={Slow entropy of some parabolic flows},
        date={2019},
     journal={Communications in Mathematical Physics},
      volume={370},
      number={2},
       pages={449\ndash 474},
}

\bib{katok1997slow}{inproceedings}{
      author={Katok, Anatole},
      author={Thouvenot, Jean-Paul},
       title={Slow entropy type invariants and smooth realization of commuting
  measure-preserving transformations},
organization={Elsevier},
        date={1997},
   booktitle={Annales de l'institut henri poincare (b) probability and
  statistics},
      volume={33},
       pages={323\ndash 338},
}

\bib{Kolmogorov}{article}{
      author={Kolmogorov, A.~N.},
       title={New metric invariant of transitive dynamical systems and
  endomorphisms of {L}ebesgue spaces},
        date={1958},
     journal={Dokl. Russ. Acad. Sci.},
      volume={119},
       pages={861\ndash 864},
}

\bib{macwilliams1977theory}{book}{
      author={MacWilliams, Florence~Jessie},
      author={Sloane, Neil James~Alexander},
       title={The theory of error-correcting codes},
   publisher={Elsevier},
        date={1977},
      volume={16},
}

\bib{shields1973cutting}{article}{
      author={Shields, Paul},
       title={Cutting and independent stacking of intervals},
        date={1973},
     journal={Mathematical systems theory},
      volume={7},
      number={1},
       pages={1\ndash 4},
}

\bib{Sinai2}{article}{
      author={Sinai, Ya.~G.},
       title={On the notion of entropy of a dynamical system},
        date={1959},
     journal={Dokl. Russ. Acad. Sci.},
      volume={124},
       pages={768\ndash 771},
}

\end{biblist}
\end{bibdiv}


\end{document}